\documentclass[preprint,12pt]{elsarticle}

\def\OO{{\rm O}}
\def\UU{{\rm U}}
\def\T{{\rm T}}
\def\H{{\rm H}}
\def\R{{\mathbb R}}
\def\CC{{\mathbb C}}
\def\rank{{\rm rank}}
\def\diag{{\rm diag}}
\newtheorem{theorem}{Theorem}

\newtheorem{lemma}{Lemma}

\usepackage{amssymb}

\begin{document}

\begin{frontmatter}



\title{A note on the hyperbolic singular value decomposition without hyperexchange matrices}


\author[mymainaddress,mysecondaryaddress]{D. S. Shirokov}
\ead{dm.shirokov@gmail.com}

\address[mymainaddress]{National Research University Higher School of Economics, 101000 Moscow, Russia}
\address[mysecondaryaddress]{Institute for Information Transmission Problems of Russian Academy of Sciences, 127051 Moscow, Russia}

\begin{abstract}
We present a new formulation of the hyperbolic singular value decomposition (HSVD) for an arbitrary complex (or real) matrix without hyperexchange matrices and redundant invariant parameters. In our formulation, we use only the concept of pseudo-unitary (or pseudo-orthogonal) matrices. We show that computing the HSVD in the general case is reduced to calculation of eigenvalues, eigenvectors, and generalized eigenvectors of some auxiliary matrices. The new formulation is more natural and useful for some applications. It naturally includes the ordinary singular value decomposition.
\end{abstract}


\end{frontmatter}


\section{Introduction}
\label{sec:1}

This paper contains two new results. The first result is the presentation of a new formulation of the HSVD for an arbitrary complex (or real) matrix (see Theorems \ref{thNew} or \ref{thNew2}) without hyperexchange matrices and redundant invariant parameters. We use only the concept of pseudo-unitary (or pseudo-orthogonal) matrices. In the standard formulation of the HSVD (see Theorem \ref{thZha}), the matrix $V$ is hyperexchange with five parameters $j$, $l$, $t$, $k$, and $s$, some of which, as it turns out in this paper, are redundant. We obtain a new formulation of the HSVD (with three parameters $j$, $l$, and $t$) without the use of hyperexchange matrices for the general case. The second result is the presentation of the relation between the HSVD and the generalized eigenvalue problem. Namely, we prove that computing the HSVD in the general case is reduced to calculation of eigenvalues, eigenvectors, and generalized eigenvectors of some auxiliary matrices (see Theorem \ref{lem1}).

The paper is organized as follows. In section \ref{sec:2}, we present the well-known formulation of the HSVD with some remarks. In section \ref{sec:3}, we discuss that replacing hyperexchange matrices by corresponding pseudo-unitary (or pseudo-orthogonal) matrices in the standard formulation of the HSVD is not correct for the general case of an arbitrary  complex (or real) matrix. However, this is correct in the particular case of full column rank matrices. In section \ref{sec:4}, we present a new formulation of the HSVD without hyperexchange matrices and redundant invariant parameters in the general case. In section \ref{sec:5}, we discuss relation between the HSVD and the generalized eigenvalue problem. Also we show that the new formulation of the HSVD naturally includes the ordinary singular value decomposition (SVD). The conclusions follow in section \ref{sec:6}.

\section{On the standard formulation of the HSVD with some remarks}
\label{sec:2}

We denote the identity matrix of size $n$ by $I=I_n=\diag(1, \ldots, 1)$ and the diagonal matrix with $+1$ appearing $p$ times followed by $-1$ appearing $q$ times on the diagonal by $J=J_{m}=\diag(I_p, -I_q)$, $p+q=m$. In the current paper, we give all statements for the complex case. All statements will be correct if we replace complex matrices by the corresponding real matrices, the operation of Hermitian conjugation ${}^\H$ by the operation of transpose ${}^\T$, the following unitary-like groups ($m=p+q$)
\begin{eqnarray}
\UU(n)=\{A\in\CC^{n\times n}, A^\H A=I\},\quad\UU(p,q)=\{A\in\CC^{m\times m}, A^\H J A=J\}\label{U}
\end{eqnarray}
by the corresponding orthogonal-like groups
\begin{eqnarray*}
\OO(n)=\{A\in\R^{n\times n}, A^\T A=I\},\quad\OO(p,q)=\{A\in\R^{m\times m}, A^\T J A=J\}.
\end{eqnarray*}
One calls the group $\OO(p,q)$ a pseudo-orthogonal group, an indefinite orthogonal group, or a group of $J$-orthogonal matrices \cite{Hall}, \cite{Higham}. There are also various names of the group $\UU(p,q)$: a pseudo-unitary group, an indefinite unitary group, a group of $J$-unitary matrices, a group of hypernormal matrices \cite{Bojan2}.

The most general version of the hyperbolic singular value decomposition (HSVD) is given in \cite{Zha} by H.~Zha.

\begin{theorem}[\cite{Zha}]\label{thZha} Assume $J=\diag(I_p, -I_q)$, $p+q=m$. For an arbitrary matrix\footnote{Note that in Theorems \ref{thZha}, \ref{thNew}, and Lemma \ref{rem5}, we use the same notation for the dimension $n\times m$ (not more standard $m\times n$) of a rectangular matrix as in Zha's work.
} $A\in\CC^{n\times m}$, there exist matrices $U\in\UU(n)$ and $V\in\CC^{m\times m}$,
\begin{eqnarray}
V^\H J V=\hat{J}:=\diag(-I_j, I_j, -I_t, I_{l-t}, I_s, -I_{k-s}),\label{ks}
\end{eqnarray}
such that
\begin{eqnarray}
A=U\Sigma  V^\H,\qquad
\Sigma =\left(
      \begin{array}{cccc}
        I_j & I_j & 0 & 0 \\
        0 & 0 & D_{l} & 0 \\
        0 & 0 & 0 & 0 \\
      \end{array}
    \right)\in\R^{n\times m},\label{hsvd}
\end{eqnarray}
where $D_l\in\R^{l \times l}$ is a diagonal matrix with all positive diagonal elements, which are uniquely determined. Here we have
$$j=\rank (A)-\rank (AJA^\H),\qquad l=\rank (AJA^\H),$$
$t$ is the number of negative eigenvalues of the matrix $AJA^\H$.
\end{theorem}

\noindent{\bf Remark 1.} \label{rem1}  Note that the statement of Theorem \ref{thZha} contains parameters $j$, $l$, $t$, $k$, and $s$. Prof. H. Zha in his work \cite{Zha} (see Remark 6) says that there are four important HSVD parameters $j$, $l$, $k$, $s$ and does not concretize who $k$ and $s$ are in (\ref{ks}). In our opinion, it is more correct to say about three (not four) invariants $j$, $l$, and $t$ (or, alternatively, $j$, $l$, and $s$), which we mention in Theorem \ref{thZha}. The numbers $k$ and $s$ are uniquely determined by $j$, $l$, and $t$:
\begin{eqnarray}
&&k=m-2j-l=m -2\rank (A)+\rank(A JA^\H),\label{kss1}\\
&&s=p-j-l+t=p-\rank (A)+t.\label{kss2}
\end{eqnarray}
Because of the law of inertia the number $p$ of $+1$ and the number $q$ of $-1$ in the matrices $J$ and $\hat{J}$ are the same. Using $j+l-t+s=p$, we get (\ref{kss2}). For determining $k$, we have $2j+l+k=m$ and obtain (\ref{kss1}).

Later we will see that a new formulation of the HSVD (Theorems \ref{thNew} and \ref{thNew2}) does not contain parameters $k$ and $s$. Thus there are three important HSVD parameters: $j$, $l$, and $t$, which depend on $A$ and $J$. The numbers $j$, $l$ and $t$ with the diagonal elements of the matrix $D$ uniquely determine the HSVD for fixed $p$, $m$, and $n$. At the same time, the matrices $U$ and $V$ are not uniquely determined in the HSVD.

\bigskip

Positive numbers on the diagonal of the matrix $D_l$ (the number of them equals $l$) and zeros on the continuation of this diagonal in the matrix $\Sigma$ (the number of such zeros equals $\min(m-2j, n-j)-l$) are called \emph{hyperbolic singular values}. Thus the number of hyperbolic singular values equals $\min(m-2j, n-j)$ in the general case.

The first formulation of the HSVD was presented by R.~Onn, A.~O.~Steinhardt and A.~W.~Bojanczyk in \cite{Bojan0} for the particular case $m\geq n$, $\rank (A J A^\H)=\rank (A)=n$ (the notation as in Theorem \ref{thZha}). In this particular case, $j=0$ and the matrix $\Sigma$ is diagonal with all positive diagonal elements. In \cite{Bojan}, the same three authors formulate the statement for a slightly more general case of arbitrary $m$ and $n$, $\rank (AJA^\H)=\rank (A)=\min(m,n)$. In the third work of the same authors \cite{Bojan2}, there is a generalization of the HSVD to the case $\rank(AJA^\H)<\rank (A)$. This generalization uses complex entries of the matrix $\Sigma$. H.~Zha \cite{Zha} indicated that this generalization seems rather unnatural and presented another generalization using only real entries of the matrix $\Sigma$. We discuss this generalization above (see Theorem \ref{thZha}). B.~C.~Levy \cite{Levi} presented the statement of Zha's result in another form using another proof. At the same time, Levy's statement is weaker than Zha's statement: there are additional arbitrary diagonal matrices instead of the identity matrices $I_j$ in the matrix $\Sigma$; there is no explicit form of the matrix $\hat{J}$ (like (\ref{ks}) in Theorem \ref{thZha}); only the case $m\geq n$ is considered. Note interesting results of S.~Hassi \cite{Hassi}, B.~N. Parlett \cite{Par}, and V.~\v{S}ego \cite{Sego}, \cite{Sego2} on other generalizations of SVD to the hyperbolic case.
In this paper, we give a generalization of Theorem \ref{thZha} without using matrices of type (\ref{ks}), which are called hyperexchange matrices.

\section{On hyperexchange matrices and the HSVD}
\label{sec:3}

In \cite{Bojan0}, a complex matrix $A$ with the condition
\begin{eqnarray}
A^\H J A=\hat{J},\label{hyper}
\end{eqnarray}
where $J=J_{m}=\diag(I_p, -I_q)$, $p+q=m$, and $\hat{J}=\hat{J}_m$ is a diagonal matrix with entries $\pm 1$ in some order, is called a hyperexchange matrix\footnote{One can find another definition of a hyperexchange matrix: $A J A^\H=\hat{J}$ (see \cite{Levi}). The second definition is not equivalent to the first one (\ref{hyper}). Multiplying both sides of (\ref{hyper}) on the left by $A \hat{J}$, and on the right by $A^{-1}J$, we get $A\hat{J} A^\H=J$, which differs from the second definition. Note that the matrix $B=A^{-1}$ satisfies $B J B^\H=\hat{J}$.}. In the particular case $J=\hat{J}$, $A$ becomes a $J$-unitary matrix $A^\H J A=J$ (or, equivalently, $AJA^\H=J$).
$J$-unitary matrices are more natural and useful for different applications. In the next section, we give a generalization of Theorem \ref{thZha} using only $J$-unitary matrices, without using hyperexchange matrices. Let us note the following fact.

\begin{lemma}\label{rem5} If we replace $V^\H J V= \hat{J}$ by $V^\H J V= J$ in Theorem \ref{thZha}, then the statement of Theorem \ref{thZha} will not be correct in the general case. In other words, we can not change the condition for matrix $V$ from hyperexchange to $J$-unitary in the formulation of Theorem \ref{thZha} in the general case.
\end{lemma}
\begin{proof} Let us give a counterexample for the real case $A\in\R^{n\times m}$:
$$A_{1\times 2}=\left(
                  \begin{array}{cc}
                    0 & 1 \\
                  \end{array}
                \right),\qquad
                J=\diag(1, -1),\qquad n=1,\qquad m=2.
$$
We have $\rank (A)=1$ and $\rank (A J A^\T)=1$. Let us prove that there are no matrices $D$, $U$, and $V$ of the following form
\begin{eqnarray}
&&D=\left(
                  \begin{array}{cc}
                    d & 0 \\
                  \end{array}
                \right)\in\R^{1\times 2},\quad
U=\left(
                \begin{array}{c}
                  u \\
                \end{array}
              \right)\in\R^{1\times 1},\quad
U^\T U= 1,\nonumber\\
&&V=\left(
                                    \begin{array}{cc}
                                      v_{11} & v_{12} \\
                                      v_{21} & v_{22} \\
                                    \end{array}
                                  \right)\in\R^{2\times 2},\quad
V^\T J V= J\nonumber
\end{eqnarray}
such that $A=U D V^\T$.

The condition $V^\T J V=J$ is equivalent to
$$v_{11}^2=1+v_{21}^2,\qquad v_{12}^2=1+v_{22}^2,\qquad v_{11}v_{12}=v_{21}v_{22}.$$
We obtain
$$\left(
    \begin{array}{cc}
      0 & 1 \\
    \end{array}
  \right)
=\left(
   \begin{array}{c}
     u \\
   \end{array}
 \right)
\left(
  \begin{array}{cc}
    d & 0 \\
  \end{array}
\right)
\left(
                \begin{array}{cc}
                  v_{11} & v_{21} \\
                  v_{12} & v_{22} \\
                \end{array}
              \right),
$$
i.e. $u d v_{11}=0$ and $u d v_{21}=1$.
Using $d\neq 0$ and $u\neq 0$, we get $v_{11}=0$, which is a contradiction to $v_{11}^2=1+v_{21}^2$. $\blacksquare$
\end{proof}

\noindent{\bf Remark 2.}\label{rem6}
If we add condition that $A_{n\times m}$ is a full column rank matrix (we have also $n \geq m$ and $j=0$ in this case) to the formulation of Theorem \ref{thZha}, then we can replace condition $V^\H J V= \hat{J}$ by $V^\H J V= J$ and the statement of the theorem will be correct. This particular case is usually considered in the literature (see, for example, \cite{Slapnicar, Novakovic, Politi}). In this section, we try to distinguish the general case and the particular cases for the convenience of the reader.

\bigskip

The counterexample above shows us that we must use the concept of hyperexchange matrices in the formulation of Theorem \ref{thZha}. However, in the next section, we give a generalization of Theorem \ref{thZha} without using the concept of hyperexchange matrices for the general case.

\section{A new formulation of the HSVD}
\label{sec:4}

\begin{theorem}\label{thNew} Assume $J=\diag(I_p, -I_q)$, $p+q=m$. For an arbitrary matrix $A\in\CC^{n\times m}$, there exist $U\in\UU(n)$ and $V\in\UU(p, q)$
such that
\begin{eqnarray}
A=U\Sigma V^\H,\label{new1}
\end{eqnarray}
where
\begin{eqnarray}
\Sigma=\lefteqn{\quad\underbrace{\phantom{
      \begin{array}{ccc}
        P_{l-t} & 0 & 0 \\
        0 & 0 & 0  \\
        0 & I_j & 0  \\
        0 & 0 & 0 \\
      \end{array}
}}_p\underbrace{\phantom{
      \begin{array}{ccc}
        Q_{t} & 0 & 0 \\
        0 & 0 & 0  \\
        0 & I_j & 0  \\
        0 & 0 & 0 \\
      \end{array}
}}_q\quad}
\left(
      \begin{array}{ccc|ccc}
        P_{l-t} & 0 & 0 & 0 & 0 & 0\\
        0 & 0 & 0 & Q_t & 0 & 0 \\
        0 & I_j & 0 & 0 & I_j & 0 \\
        0 & 0 & 0 & 0 & 0 & 0\\
      \end{array}
    \right)\in\R^{n\times m},\label{D}
\end{eqnarray}
where the first block has $p$ columns and the second block has $q$ columns, $P_{l-t}$ and $Q_{t}$ are diagonal matrices of corresponding dimensions $l-t$ and $t$ with all positive uniquely determined diagonal elements (up to a permutation).

Moreover, choosing $U$, one can swap rows of the matrix $\Sigma$. Choosing $V$, one can swap columns in individual blocks but not across blocks. Thus we can always arrange diagonal elements of the matrices $P_{l-t}$ and $Q_t$ in decreasing (or ascending) order\footnote{Alternatively, we can change the order of the first $l$ rows of the matrix $\Sigma$ and obtain all nonzero elements of the first $l$ rows of the matrix $\Sigma$ in decreasing (or ascending) order.}.

Here we have
$$j=\rank (A)-\rank (AJA^\H),\qquad l=\rank (AJA^\H),$$
and $t$ is the number of negative eigenvalues of the matrix $AJA^\H$ (note that $l-t$ is the number of positive eigenvalues of the matrix $AJA^\H$).
\end{theorem}

\begin{proof} Let us use the statement of Theorem \ref{thZha} with hyperexchange matrix $V$ satisfying
\begin{eqnarray}
V^\H J V=\hat{J}=\diag(-I_j, I_j, -I_t, I_{l-t}, I_s, -I_{k-s}).\label{hy}
\end{eqnarray}
It is not difficult to show that hyperexchange matrices and $J$-unitary matrices are closely connected: for an arbitrary hyperexchange matrix $V$, there exists a permutation matrix $S^\T=S^{-1}$ such that $F:=VS$ is $J$-unitary.

From the law of inertia, it follows that matrices $J$ and $\hat{J}$ have the same numbers of $1$ and $-1$ on the diagonal. It means that these two matrices are connected with the aid of some permutation matrix $S$: $\hat{J}=S J S^\T$. Let us remind the reader that a permutation matrix has exactly one nonzero element, equal to $1$, in each column and in each row. A permutation matrix is orthogonal $S^\T S=I$. We get $SJ S^\T=V^\H J V$, i.e. $(VS)^\H J (VS)=J$ and $F=VS$ is $J$-unitary.

From $A=U\Sigma V^\H$ (\ref{hsvd}), we obtain $A=U\Sigma S F^\H$. Multiplying the matrix $\Sigma$ on the right by $S$, we change the order of its columns. Using $S^\T \hat{J} S=J$ and the explicit form of the matrix $\hat{J}$ (\ref{hy}), we get the explicit form of the matrix $S$:
$$S=\lefteqn{\quad\underbrace{\phantom{
      \begin{array}{ccc}
        0 & 0 & 0 \\
        0 & I_j & 0 \\
        0 & 0 & 0  \\
        I_{l-t} & 0 & 0  \\
        0 & 0 & I_s  \\
        0 & 0 & 0 \\
      \end{array}
}}_p\underbrace{\phantom{
      \begin{array}{ccc}
         0 & I_j & 0 \\
         0 & 0 & 0 \\
        I_t & 0 & 0 \\
         0 & 0 & 0 \\
        0 & 0 & 0 \\
         0 & 0 & I_{k-s} \\
      \end{array}
}}_q\quad}\left(
      \begin{array}{ccc|ccc}
        0 & 0 & 0 & 0 & I_j & 0 \\
        0 & I_j & 0 & 0 & 0 & 0 \\
        0 & 0 & 0 & I_t & 0 & 0 \\
        I_{l-t} & 0 & 0 & 0 & 0 & 0 \\
        0 & 0 & I_s & 0 & 0 & 0 \\
        0 & 0 & 0 & 0 & 0 & I_{k-s} \\
      \end{array}
    \right)
.$$
Then we calculate the explicit form of the matrix $\Sigma S$, where $\Sigma$ is from (\ref{hsvd}):
\begin{eqnarray}\Sigma S=\lefteqn{\quad\underbrace{\phantom{
       \begin{array}{ccc}
               0 & I_j & 0\\
               0 & 0 & 0 \\
               P_{l-t} & 0 & 0 \\
               0 & 0 & 0\\
             \end{array}
}}_p\underbrace{\phantom{
      \begin{array}{ccc|ccc}
               0 & I_j & 0 \\
               Q_t & 0 & 0 \\
                0 & 0 & 0 \\
                0 & 0 & 0\\
             \end{array}
}}_q\quad}\left(
             \begin{array}{ccc|ccc}
               0 & I_j & 0 & 0 & I_j & 0 \\
               0 & 0 & 0 & Q_t & 0 & 0 \\
               P_{l-t} & 0 & 0 & 0 & 0 & 0 \\
               0 & 0 & 0 & 0 & 0 & 0\\
             \end{array}
           \right)\label{TT}
,\end{eqnarray}
where $D_l=\diag(Q_t, P_{l-t})$.

We can multiply the matrix (\ref{TT}) by an arbitrary permutation matrix $S'$ on the left $S'\Sigma$ because $S'\in\OO(n)$. Thus we can swap rows of the matrix (\ref{TT}). We can multiply the matrix (\ref{TT}) on the right by an arbitrary permutation matrix of the form
$$\left(
      \begin{array}{c|c}
        S_1 & 0 \\ \hline
        0 & S_2  \\
      \end{array}
    \right)\in\OO(p,q),$$
where $S_1$ and $S_2$ are arbitrary permutation matrices of order $p$ and $q$ respectively. Thus we can swap columns in individual blocks but not across blocks. Finally, we obtain the explicit form of the new matrix $\Sigma$ (\ref{D}), where $P_{l-t}$ and $Q_{t}$ are diagonal matrices with all positive uniquely determined diagonal elements in decreasing (or ascending) order. $\blacksquare$
\end{proof}

Note that there are no indices $k$ and $s$ in the formulation of Theorem \ref{thNew} (but they are in the formulation of Theorem \ref{thZha}, see Remark 1). These indices do not have any important information on the HSVD.

Note that we can change $V^\H$ to $V$ in (\ref{new1}), because if $V\in\UU(p,q)$, then $V^\H\in\UU(p,q)$. Since analogous reasoning is not correct for hyperexchange matrices, we can not do the same in (\ref{hsvd}).

For the convenience of the reader, let us give a reformulation of Theorem \ref{thNew} to the case when a $J$-unitary matrix is on the left side and a unitary matrix is on the right side (as in \cite{Levi} but now without using hyperexchange matrices and for the general case).

\begin{theorem}\label{thNew2} Assume $J=\diag(I_p, -I_q)$, $p+q=m$. For an arbitrary matrix $B\in\CC^{m\times n}$, there exist
$U_0\in\UU(n)$ and $V_0\in\UU(p, q)$
such that
\begin{eqnarray}
V_0^\H B U_0=\Sigma,\label{new2}
\end{eqnarray}
where
\begin{eqnarray}
\Sigma=\left(
      \begin{array}{cccc}
        P_{l-t} & 0 & 0 & 0 \\
        0 & 0 & I_j & 0 \\
        0 & 0 & 0 & 0 \\
        \hline
        0 & Q_t & 0 & 0 \\
        0 & 0 & I_j & 0 \\
        0 & 0 & 0 & 0\\
      \end{array}
    \right)\!\!\!\!\!\!\!\!\!
    \begin{array}{c}
      \left.
      \begin{array}{c}
          \\
         \\
          \\
      \end{array}
    \right\}p \\
      \left.
      \begin{array}{c}
          \\
         \\
          \\
      \end{array}
    \right\}q
    \end{array}
    \in\R^{m\times n},\label{DD2}
\end{eqnarray}
where the first block has $p$ rows and the second block has $q$ rows, $P_{l-t}$ and $Q_t$ are diagonal matrices of corresponding dimensions $l-t$ and $t$ with all positive uniquely determined diagonal elements (up to a permutation).

Moreover, choosing $U_0$, one can swap columns of the matrix $\Sigma$. Choosing $V_0$, one can swap rows in individual blocks but not across blocks. Thus we can always arrange diagonal elements of the matrices $P_{l-t}$ and $Q_t$ in decreasing (or ascending) order\footnote{Alternatively, we can change the order of the first $l$ columns of the matrix $\Sigma$ and obtain all nonzero elements of the first $l$ columns of the matrix $\Sigma$ in decreasing (or ascending) order.}.

Here we have
$$j=\rank (B)-\rank (B^\H JB),\qquad l=\rank (B^\H JB),$$
and $t$ is the number of negative eigenvalues of the matrix $B^\H JB$ (note that $l-t$ is the number of positive eigenvalues of the matrix $B^\H JB$).
\end{theorem}

\begin{proof} Using $A=U\Sigma V^\H$ (\ref{hsvd}), we get $A^\H=V \Sigma^\T U^\H$. Multiplying both sides on the left by $V^{-1}$, and on the right by $U$, we get $V^{-1}A^\H U= \Sigma^\T$. Using the notation $B=A^\H$, $V_0^\H=V^{-1}\in\UU(p,q)$, $U_0=U\in\UU(n)$, we obtain the statement of the theorem. $\blacksquare$
\end{proof}

\section{Computing the HSVD}
\label{sec:5}

The new formulation of the HSVD allows us to compute the HSVD in the general case. In Theorem \ref{lem1}, we show that computing the HSVD is reduced to calculation of eigenvalues, eigenvectors, and generalized eigenvectors of some auxiliary matrices.

In this section, we use the formulation of the HSVD from Theorem \ref{thNew2}. For arbitrary matrix $B\in\CC^{m\times n}$, we can easily find matrices $V_0\in\UU(p,q)$, $U_0\in\UU(n)$, $\Sigma\in\R^{m\times n}$ of the form (\ref{DD2}) such that $V_0$, $B$, and $U_0$ satisfy (\ref{new2}).

\begin{theorem}\label{lem1} For the matrices $B$, $V_0$, $U_0$, and $\Sigma$ from Theorem \ref{thNew2}, we have the following equations:
\begin{eqnarray}
(B^\H J B)U_0=U_0(\Sigma^\T J \Sigma),\qquad (J B B^\H)V_0=V_0(J \Sigma \Sigma^\T).\label{calc}
\end{eqnarray}
The hyperbolic singular values of the matrix $B$ are square roots of the modules of the eigenvalues of the matrix $B^\H J B$. The columns of the matrix $U_0$ are corresponding eigenvectors of the matrix $B^\H J B$. The columns of the matrix $V_0$ are corresponding eigenvectors of the matrix $J B B^\H$ (in the case $j=0$), or corresponding eigenvectors and generalized eigenvectors of the matrix $J B B^\H$ (in the case $j\neq 0$).
\end{theorem}
\begin{proof}
From (\ref{new2}), we obtain
\begin{eqnarray}
U_0^\H B^\H V_0=\Sigma^\T.\label{w1}
\end{eqnarray}
Multiplying on the left by $U_0$, and on the right by $J\Sigma$, we get
\begin{eqnarray}
B^\H V_0J \Sigma=U_0 \Sigma^\T J \Sigma.\label{w2}
\end{eqnarray}
Using (\ref{w2}) and (\ref{new2}), we obtain the first equation from (\ref{calc}).

Multiplying (\ref{new2}) on the left by $V_0J$, and on the right by $\Sigma^\T$, we get
\begin{eqnarray}
JBU_0\Sigma^\T=V_0J\Sigma \Sigma^\T.\label{w3}
\end{eqnarray}
Using (\ref{w3}) and (\ref{w1}), we obtain the second equation from (\ref{calc}).

If we denote
$$P_{l-t}=\diag(p_1, \ldots, p_{l-t}),\qquad Q_t=\diag(q_1, \ldots, q_t),$$
then it can be easily verified that
$$\Sigma^\T J \Sigma=\diag(P^2_{l-t}, -Q^2_t, 0)=\diag(p_1^2, \ldots, p_{l-t}^2, -q_1^2, \ldots, -q_{t}^2, 0, \ldots, 0).$$
From this equation and the first equation (\ref{calc}), it follows that hyperbolic singular values of the matrix $B$ are square roots of the modules of the eigenvalues of the matrix $B^\H J B$. The columns of the matrix $U_0$ are eigenvectors of the matrix $B^\H J B$.

We have
$$J \Sigma \Sigma^\T= \lefteqn{\quad\underbrace{\phantom{
       \begin{array}{ccc}
                          P^2_{l-t} & 0 & 0\\
                          0 & I_j & 0  \\
                          0 & 0 & 0   \\ \hline
                          0 & 0 & 0  \\
                          0 & -I_j & 0  \\
                          0 & 0 & 0  \\
                        \end{array}
}}_p\underbrace{\phantom{
      \begin{array}{ccc}
                          P^2_{l-t} & 0 & 0  \\
                          0 & I_j & 0  \\
                          0 & 0 & 0  \\ \hline
                          0 & 0 & 0  \\
                          0 & -I_j & 0  \\
                          0 & 0 & 0 \\
                        \end{array}
}}_q\quad}\left(
                        \begin{array}{ccc|ccc}
                          P^2_{l-t} & 0 & 0 & 0 & 0 & 0 \\
                          0 & I_j & 0 & 0 & I_j & 0 \\
                          0 & 0 & 0 & 0 & 0 & 0  \\ \hline
                          0 & 0 & 0 & -Q^2_t & 0 & 0 \\
                          0 & -I_j & 0 & 0 & -I_j & 0 \\
                          0 & 0 & 0 & 0 & 0 & 0 \\
                        \end{array}
                      \right)\!\!\!\!\!\!\!\!\!
    \begin{array}{c}
      \left.
      \begin{array}{c}
          \\
         \\
          \\
      \end{array}
    \right\}p \\
      \left.
      \begin{array}{c}
          \\
         \\
          \\
      \end{array}
    \right\}q
    \end{array}.
$$
Using this equation and the second equation (\ref{calc}), we can find the matrix $V_0$. In the case $j=0$, the columns of the matrix $V_0$ are eigenvectors of the matrix $J B B^\H$. In the case $j\neq 0$, the columns of the matrix $V_0$, which correspond to the blocks $P^2$, $Q^2$, and zero blocks, are eigenvectors of the matrix $JBB^\H$. The remaining columns $v_i$, $w_i$, $i=1, \ldots, j$ of the matrix $V_0$, which correspond to the two blocks $I_j$, satisfy the conditions
$$(J B B^\H)v_i=(J B B^\H)w_i=v_i-w_i,\qquad i=1, \ldots, j,$$
and therefore
$$(JBB^\H)(v_i-w_i)=0,\qquad (JBB^\H)^2 v_i=(JBB^\H)^2 w_i=0,\qquad i=1, \ldots , j,$$
i.e. $v_i$, $w_i$, $i=1, \ldots, j$ are generalized eigenvectors of the matrix $J B B^\H$. $\blacksquare$
\end{proof}

\noindent{\bf Example 1.} ($j=0$) Let us consider the following example
$$B=\left(
     \begin{array}{c}
       1 \\
       2 \\
     \end{array}
   \right),\qquad J=\diag(1, -1),\qquad m=2,\qquad p=q=1,\qquad n=1.$$
In this case, we have
$$B^\T J B=-3,\qquad \rank (B)=\rank (B^\T JB)=1,\qquad j=0, \qquad l=1.$$
Since eigenvalue of the matrix $B^\T JB$ equals $-3$, it follows that $t=1$ and the hyperbolic singular value of the matrix $B$ is $\sqrt{3}$. We can choose the following matrix $U_0\in\OO(1)$, the matrix $\Sigma$ is determined uniquely:
$$\Sigma=\left(
     \begin{array}{c}
       0 \\
       \sqrt{3} \\
     \end{array}
   \right),\qquad U_0=\left(
                      \begin{array}{c}
                        1 \\
                      \end{array}
                    \right).
   $$
Using $(J B B^\T)V_0=V_0(J \Sigma \Sigma^\T)$, we get
$$
\left(
  \begin{array}{cc}
    1 & 2 \\
    -2 & -4 \\
  \end{array}
\right) V_0=V_0 \left(
              \begin{array}{cc}
                0 & 0 \\
                0 & -3 \\
              \end{array}
            \right).$$
Note that $0$ and $-3$ are eigenvalues of the matrix $JB B^\T$. Calculating eigenvectors of the matrix $J B B^\T$ and choosing correct multipliers (taking into account $V_0^\T J V_0=J$), we get
$$V_0=\left(
      \begin{array}{cc}
        \frac{-2}{\sqrt{3}} & \frac{-1}{\sqrt{3}} \\
        \frac{1}{\sqrt{3}} & \frac{2}{\sqrt{3}} \\
      \end{array}
    \right)\in\OO(1,1).
$$
Finally, we have
\begin{eqnarray}
V_0^\T B U_0=\Sigma,\qquad\qquad \left(
      \begin{array}{cc}
        \frac{-2}{\sqrt{3}} & \frac{-1}{\sqrt{3}} \\
        \frac{1}{\sqrt{3}} & \frac{2}{\sqrt{3}} \\
      \end{array}
    \right)^\T \left(
     \begin{array}{c}
       1 \\
       2 \\
     \end{array}
   \right) \left(
             \begin{array}{c}
               1 \\
             \end{array}
           \right)
   =\left(
     \begin{array}{c}
       0 \\
       \sqrt{3} \\
     \end{array}
   \right).\label{LR}
\end{eqnarray}
Note that the matrices $V_0$ and $U_0$ in (\ref{LR}) are not determined uniquely. For example, we can change the signs of these matrices at the same time.

\bigskip

\noindent{\bf Example 2.} ($j\neq 0$) Let us consider the following example
$$B=\left(
     \begin{array}{c}
       2 \\
       2 \\
     \end{array}
   \right),\qquad J=\diag(1, -1),\qquad m=2,\qquad p=q=1,\qquad n=1.$$
In this case, we have
$$B^\T J B=0,\qquad l=\rank (B^\T JB)=0,\qquad j=\rank (B)-\rank (B^\T JB)=1.$$
We have no hyperbolic singular values in this case. We can choose the following matrix $U_0\in\OO(1)$, the matrix $\Sigma$ is determined uniquely:
$$\Sigma=\left(
     \begin{array}{c}
       1 \\
       1 \\
     \end{array}
   \right),\qquad U_0=\left(
                      \begin{array}{c}
                        1 \\
                      \end{array}
                    \right).
   $$
We calculate the matrix $B \Sigma \Sigma^\T$ and its eigenvector $a_1$:
$$J \Sigma \Sigma^\T=\left(
             \begin{array}{cc}
               2 & 2 \\
               -2 & -2 \\
             \end{array}
           \right),\qquad a_1:=\left(
  \begin{array}{c}
    1 \\
    -1 \\
  \end{array}
\right).
$$
Calculating corresponding generalized eigenvectors $v_1$ and $w_1$
$$(J B B^\T)v_1=(JBB^\T)w_1=a_1$$
and choosing correct multipliers (taking into account $V_0^\T J V_0=J$), we get
$$v_1=\left(
  \begin{array}{c}
    \frac{5}{4} \\
    -\frac{3}{4} \\
  \end{array}
\right),\qquad w_1=\left(
  \begin{array}{c}
    -\frac{3}{4} \\
    \frac{5}{4} \\
  \end{array}
\right),\qquad V_0=\left(
      \begin{array}{cc}
        \frac{5}{4} & -\frac{3}{4} \\
        -\frac{3}{4} & \frac{5}{4} \\
      \end{array}
    \right)\in\OO(1,1).
$$
Finally, we have
\begin{eqnarray}
V_0^\T B U_0=\Sigma,\qquad\qquad \left(
      \begin{array}{cc}
        \frac{5}{4} & -\frac{3}{4} \\
        -\frac{3}{4} & \frac{5}{4} \\
      \end{array}
    \right)^\T \left(
     \begin{array}{c}
       2 \\
       2 \\
     \end{array}
   \right) \left(
             \begin{array}{c}
               1 \\
             \end{array}
           \right)
   =\left(
     \begin{array}{c}
       1 \\
       1 \\
     \end{array}
   \right).\nonumber
\end{eqnarray}
The matrices $V_0$ and $U_0$ are not determined uniquely.

\bigskip

\noindent{\bf Remark 3.}\label{rem11}
In the case $J=I$ ($p=m$, $q=0$), we obtain $j=0$, $t=0$, and the ordinary singular value decomposition \cite{svd0}, \cite{svd} as the particular case of Theorem \ref{thNew2} with $V_0\in\UU(m)$, $U_0\in\UU(n)$. In this case, the matrix $\Sigma$ is diagonal with all nonnegative diagonal elements. In this case, we obtain from (\ref{calc}) the well-known formulas
$$(B^\H B)U_0=U_0(\Sigma^\T \Sigma),\qquad (B B^\H)V_0=V_0(\Sigma \Sigma^\T)$$
for finding $\Sigma$, $U_0$, and $V_0$. In this case, singular values of the matrix $B$ are square roots of the eigenvalues of the positive-definite Hermitian matrices $B^\H B$ and $B B^\H$, the columns of the matrix $V_0$ are eigenvectors of the matrix $B B^\H$, and the columns of the matrix $U_0$ are eigenvectors of the matrix $B^\H B$.

\section{Conclusions}
\label{sec:6}

In this paper, we present a new formulation of the HSVD for an arbitrary complex (or real) matrix without hyperexchange matrices and redundant invariant parameters. We use only the concept of pseudo-unitary (or pseudo-orthogonal) matrices. The expressions (\ref{D}) and (\ref{DD2}) can be regarded as new useful canonical forms of an arbitrary complex (or real) matrix. We show that computing the HSVD is reduced to calculation of eigenvalues, eigenvectors and generalized eigenvectors of some auxiliary matrices.
The new formulation of the HSVD naturally includes the ordinary SVD.

In the new formulation, we have two diagonal matrices $P$ and $Q$ in (\ref{D}) instead of one diagonal matrix $D$ in (\ref{hsvd}). This fact has physical (or geometrical) meaning. The matrix $A$ may describe some tensor field, the matrices $U$ and $V$ may describe some (coordinate, gauge) transformations. The matrix $\Sigma$ describes the same tensor field, but in some new coordinate system and with a new gauge fixing. The blocks $P$ and $Q$ of the matrix $\Sigma$ describe the contributions of the tensor field to (using physical terminology for the case $p=1$ and $q=3$) ``time'' (the first $p$) and ``space'' (the last $q$) coordinates.  Such contributions depend on the number of positive $l-t$ and negative $t$ eigenvalues of the matrix $AJA^\H$ respectively. From the statement of Theorem \ref{thZha}, it is not clear why there are exactly two blocks $I_j$ in (\ref{hsvd}) in the degenerate case $j\neq 0$. From the new formulation (Theorem \ref{thNew} or \ref{thNew2}), we see the meaning of this fact: each of two blocks $I_j$ carries information about degeneration in each of two (``space'' and ``time'') blocks of the matrix $\Sigma$. We use results of this paper to generalize results on Yang-Mills equations in Euclidean space $\R^n$ \cite{YM} to the case of pseudo-Euclidean space $\R^{p,q}$ \cite{YM2} of an arbitrary dimension $p+q$. We expect further use of the HSVD in computer science \cite{B1}, \cite{Politi},  engineering \cite{N1}, image and signal processing \cite{B2}, \cite{Bojan}, and physics \cite{N2}.

\section*{Acknowledgments}

The author is grateful to Prof. N. Marchuk for fruitful discussions. This work is supported by the Russian Science Foundation (project 18-71-00010).





\end{document}